\documentclass[12pt, reqno]{amsart}
\usepackage{amsmath, amsthm, amscd, amsfonts, amssymb, graphicx, xcolor}
\usepackage[bookmarksnumbered, colorlinks, plainpages]{hyperref}
\usepackage{amsmath}
\usepackage{cases}
\usepackage{amssymb}
\usepackage{mathrsfs}
\usepackage{euscript}
\usepackage{dsfont}
\usepackage{graphicx}
\usepackage{float}
\allowdisplaybreaks[4]
\newtheorem{theorem}{Theorem}[section]

\newtheorem{lemma}{Lemma}[section]

\newtheorem{corollary}[theorem]{Corollary}
\theoremstyle{definition}

\newtheorem{conjecture}[theorem]{Conjecture}

\theoremstyle{remark}
\newtheorem{remark}[theorem]{Remark}
\numberwithin{equation}{section}

\begin{document}
	\setcounter{page}{1}
	

	
	
	\title{On the largest prime factors less than $y$ of consecutive shifted primes}

	\author[]{Zhiyuan Yang}
	
	\address{School of Mathematics,  Shandong University, Jinan 250100, Shandong, China}
	\email{\tt zhiyuan.yang@mail.sdu.edu.cn}
	
	
	

	
	
	\begin{abstract}
		For an integer $n > 1$, let $P^+(n)$ be the largest prime factor of $n$, and let \(P_y^+(n)\) denote the largest prime factor of $n$ not exceeding $y$. One of Erd\H{o}s and Turán's
		conjectures asserts that the asymptotic density of integers $n$ satisfying $P^+(n)<P^+(n+1)$ is $1/2$. Furthermore, Wang conjectures that for $x\rightarrow\infty$, one has $\#\{p\leq x:P^+(p-1)<P^+(p+1)\}\sim \pi(x)/2$. In this paper, we show the following results. 
		For any $3\leq y\ll x^{o(1)}$, for $x\rightarrow\infty$, we have
		\begin{align*}
			&\#\{n\leq x:P^+_y(n)<P^+_y(n+1)\}\sim \frac{1}{2}x,  \\&\#\{p\leq x:P^+_y(p-1)<P^+_y(p+1)\}\sim \frac{1}{2}\pi(x),
		\end{align*}
		where the first result extends a previous range $3\leq y\ll \exp(\log x/(100\log\log x))$ due to Rivat (2001).
		For any $0<\alpha<17/32$, let $y=x^\alpha$. Then there exists $h(\alpha)>0$ such that
		\begin{align*}
			\#\{p\leq x:P_y^+(p-1)<P_y^+(p+1)\}\geq(h(\alpha)+o(1))\pi(x).
		\end{align*} 
		In particular, the function $h$ satisfies $\lim_{\alpha\rightarrow 0^+}h(\alpha)=1/2$. Similar result also holds for $\#\{n\leq x:P_y^+(n)<P_y^+(n+1)\}$, which improves a previous result $\lim_{\alpha\rightarrow 0^+}h(\alpha)=1/4$ due to Wang (2019).
		\newline
		\noindent 
	\end{abstract} \maketitle
	

	\section{Introduction}
	
	Let $P^+(n)$ denote the largest prime factor of $n$ with the convention that $P^+(1)=1$. In the 1930s, Erd\H{o}s and Turán [9] formulated the following conjecture in a correspondence.
	\begin{conjecture}
		(Erd\H{o}s-Turán) For $x\rightarrow\infty$, we have
		\begin{align*}
			\#\{n\leq x:P^+(n)<P^+(n+1)\}\sim \frac{1}{2}x.
		\end{align*}
	\end{conjecture}
	In 1978, Erd\H{o}s and Pomerance [4] first proved that there exists a positive asymptotic density of integers $n$ such that $P^+(n)<P^+(n+1)$. In fact, they showed that\begin{align*}
		\#\{n\leq x:P^+(n)<P^+(n+1)\}>0.0099x.
	\end{align*}
	In 2005, the asymptotic density was improved to 0.05544 by La Bretèche, Pomerance and Tenenbaum [2], and to 0.05866 by Fouvry's arguments in ``Further remarks" of the same paper. Later, the constant 0.05866 was improved successively to 0.1063 and 0.1356 by Wang in [12, 13]. In 2025, Lü and Wang [7] improved this constant to 0.2017. Very recently, the author's preprints [17, 18] showed that the asymptotic density is larger than 0.280 and 0.299. 
	
	On the other hand, this problem has generated a long sequence of unconditional density advances, together with logarithmic-density, averaged-shift, and conditional forms [6, 8, 10, 11, 16]. We mainly introduce the work of Rivat. In 2001, Rivat in [8] proved a $P^+_y(n)$-version of the Erd\H{o}s-Turán conjecture for some small $y$, where $P^+_y(n)$ denotes the greatest prime factor $p$ of $n$ satisfying $p\leq y$. In fact, Rivat's results implies that for $3\leq y\leq \exp(\log x/(100\log\log x))$, we have
	\begin{align*}
		\#\{n\leq x:P^+_y(n)<P^+_y(n+1)\}\sim\frac{1}{2}x.
	\end{align*}
	In [13], Wang give a generalized version of Conjecture 1.1:
	For $x\rightarrow\infty$ and $3\leq y\leq x$, one has
	\begin{align*}
		\#\{n\leq x:P_y^+(n)<P_y^+(n+1)\}\sim \frac{1}{2}x\ ?
	\end{align*}
	In the same paper, Wang proved that for any $0<\alpha\leq 1$ and $y=x^\alpha$, there exists $C(\alpha)>0$ such that
	\begin{align*}
		\#\{n\leq x:P^+_y(n)<P^+_y(n+1)\}\geq (C(\alpha)+o(1))x.
	\end{align*}
	However, the value of $C(\alpha)$ obtained in that case decreases to $0$ as $\alpha$ decreases to $0$, which is contrary to the expectation that it should be easier to solve the problem as $\alpha$ becomes smaller.
	In 2019, Wang [15] find a new function $C_0(\alpha)>0$ for $0<\alpha\leq 1/6$ such that
	\begin{align*}
		\#\{n\leq x:P^+_y(n)<P^+_y(n+1)\}\geq (C_0(\alpha)+o(1))x,
	\end{align*}
	where the function $C_0(\alpha)$ satisfies $\lim_{\alpha\rightarrow 0^+}C_0(\alpha)=1/4$. However, the range \(x^{o(1)}\gg y>\exp(\log x/(100\log\log x))\) remains unexplored. In this paper, we fill this gap and improve the above limit to $1/2$.

	We are also interested in the analogue of Conjecture 1.1 for prime numbers. A famous question is the twin prime conjecture.
	\begin{conjecture}
		(twin prime) For $x\rightarrow\infty$, we have
		\begin{align*}
			\#\{p\leq x:P^+(p)<P^+(p+2)\}\rightarrow \infty.
		\end{align*}
	\end{conjecture}
	In [14], Wang gave a symmetric‑form conjecture for primes.
	\begin{conjecture}
		(Wang) For $x\rightarrow\infty$, we have
		\begin{align*}
			\#\{p\leq x:P^+(p-1)<P^+(p+1)\}\sim \frac{1}{2}\pi(x).
		\end{align*}
	\end{conjecture}
	In [14], under the Elliott-Halberstam conjecture, Wang proved that there exists a positive asymptotic density of integers $p$ such that $P^+(p-1)<P^+(p+1)$.
	In fact, he showed the asymptotic density is larger than $0.1779$. If we replace $p$ by almost‑primes $P_2$, this problem has been independently solved by the author and Wang (private communication). In fact, for any $0<\alpha\leq 1$ and $y=x^\alpha$, there exists $z>x^{\varepsilon}$ such that
	\begin{align*}
		\#\{P_2\leq x:P^-(P_2)>z,P^+_y(P_2-1)<P^+_y(P_2+1)\}\gg \pi(x),
	\end{align*}
	where $P^-(n)$ denote the 
	smallest prime factor of $n$. Very recently, Ding-Wang's preprint [3] showed a positive density result:
	\begin{align*}
		\#\{p\leq x:P^+(p-1)<P^+(p+1)\}\geq 0.000012\pi(x).
	\end{align*}
	This is a very important work for this problem.

	Naturally, we may conjecture that
	for $x\rightarrow\infty$ and $3\leq y\leq x$, one has
	\begin{align*}
		\#\{p\leq x:P_y^+(p-1)<P_y^+(p+1)\}\sim \frac{1}{2}\pi(x).
	\end{align*}
	In this paper, we obtain the following results. Statements like $f(x)\ll x^{o(1)}g(x)$ should be read as $\forall \varepsilon>0$, $f(x)\ll_\varepsilon x^\varepsilon g(x)$.
	
	\begin{theorem}
		For $3\leq y\ll x^{o(1)}$, we have
		\begin{align*}
			\#\{p\leq x:P_y^+(p-1)<P_y^+(p+1)\sim \frac{1}{2}\pi(x).
		\end{align*}
		The lower bound is also true for the pattern $P_y^+(p-1)>P_y^+(p+1)$.
	\end{theorem}
	
	\begin{theorem}
		For $3\leq y\ll x^{o(1)}$, we have
		\begin{align*}
			\#\{n\leq x:P_y^+(n)<P_y^+(n+1)\sim \frac{1}{2}x.
		\end{align*} 
		The lower bound is also true for the pattern $P_y^+(n)>P_y^+(n+1)$.
	\end{theorem}
	In particular, for any $\varepsilon>0$, there exists $\delta=\delta(\varepsilon)>0$ such that the following holds: 
	\begin{align*}
		&\Big|\#\{n\leq x:P_y^+(n)<P_y^+(n+1)-\frac{x}{2}\Big|\leq \varepsilon x,\\&\Big|\#\{p\leq x:P_y^+(p-1)<P_y^+(p+1)-\frac{\pi(x)}{2}\Big|\leq \varepsilon \pi(x),
	\end{align*}
	for $3\leq y\leq x^{\delta}$ and $x\rightarrow\infty$.
	
	\begin{theorem}
		For any $0<\alpha<17/32$, let $y=x^\alpha$. Then there exists $h(\alpha)>0$ satisfying
		\begin{align*}
			\#\{p\leq x:P_y^+(p-1)<P_y^+(p+1)\geq(h(\alpha)+o(1))\pi(x).
		\end{align*} 
		In particular, the function $h$ satisfies
		\begin{align*}
			\lim_{\alpha\rightarrow 0^+}h(\alpha)=\frac{1}{2},\quad \lim_{\alpha\rightarrow (17/32)^-}h(\alpha)=0.
		\end{align*}
	\end{theorem}
	We also know that the above result holds when $\alpha$ is near $1$ due to Ding-Wang [3]. However, their proof makes use of quite powerful techniques. The author could not use their approach to fill the gap for $\alpha\in [17/32,1)$, and this remains a problem worthy of consideration.
	\begin{corollary}
		For $3\leq y\leq x$, we have
		\begin{align*}
			\#\{n\leq x:P_y^+(n)<P_y^+(n+1)\gg x.
		\end{align*}
		For $3\leq y\leq x^{17/32-\varepsilon}$, we have
		\begin{align*}
			\#\{p\leq x:P_y^+(p-1)<P_y^+(p+1)\gg_{\varepsilon}\pi(x).
		\end{align*}
		For $0<\alpha\leq 1$, there exists $C(\alpha)>0$ such that
		\begin{align*}
			\#\{n\leq x:P_y^+(n)<P_y^+(n+1)\geq (C(\alpha)+o(1)) x,
		\end{align*}
		where the function $C$ satisfies  $\lim_{\alpha\rightarrow 0^+}C(\alpha)=1/2$.
	\end{corollary}
	The all above results are also true for the pattern $P_y^+(n)>P_y^+(n+1)$ or $P_y^+(p-1)>P_y^+(p+1)$.
	
	\section{Lemmas}
	
	The following lemma is an important tool for proving Theorem 1.6, and it is derived from Selberg's upper-bound sieve.
	\begin{lemma}
		Let $g$ be a natural number and let $a_i, b_i $ ($i=1,2,\cdots, g$) be integers satisfying
		\begin{align*}
			E:=\prod_{i=1}^g a_i\prod_{1\leq r<\leq g}(a_r b_s-a_sb_r)\neq 0.
		\end{align*}
		Let $\omega(p)$ denote the number of solutions in n modulo $p$ of 
		\begin{align*}
			\prod_{i=1}^g(a_in+b_i)\equiv 0\mkern-15mu\pmod{p},
		\end{align*}
		and suppose that
		\begin{align*}
			\omega(p)<p, \ \text{for all}\ p.
		\end{align*}
		If the real numbers $y$ and $z$ satisfy $1<y\leq z$, then 
		\begin{align*}
			&\#\{n:z-y<n\leq z, a_in+b_i \ \text{prime for}\ i=1,2,\cdots, g\}\\&\leq 2^gg!\prod_p\left(1-\frac{\omega(p)-1}{p-1}\right)\left(1-\frac{1}{p}\right)^{-g+1}\frac{y}{\log^g y}\Bigg(1+O\left(\frac{\log\log (3y)+\log\log |3E|}{\log y}\right)\Bigg),
		\end{align*}
		where the constant implied by the $O$-symbol depends at most on $g$.
	\end{lemma}
	\begin{proof} This is [5, Theorem 5.7].
	\end{proof}

	\begin{lemma}
		For any given positive constant $A>0$, for $q\leq (\log x)^A$ and $(a,q)=1$, we have
		\begin{align*}
			\pi(x;q,a)=\sum_{\substack{p\leq x\\p\equiv a\mkern-15mu\pmod{q}}}1=\frac{\pi(x)}{\varphi(q)}+O_A(x\exp(-c\sqrt{\log x})),
		\end{align*}
		where $c=c(A)>0$.
	\end{lemma}	
	\begin{proof}
		This is the Siegel-Walfisz theorem.
	\end{proof}
	
	\begin{lemma}
		For any given positive constant $A>0$, there exists a constant $B=B(A)>0$ such that \begin{align*}
			\sum_{q\leq x^{1/2}/(\log x)^B}\max_{y\leq x}\max_{(a,q)=1}\Big|\sum_{\substack{p\leq y\\p\equiv a\mkern-15mu\pmod{q}}}1-\frac{\pi(y)}{\varphi(q)}\Big)\Big|\ll\frac{x}{(\log x)^A},
		\end{align*}
		where the constant implied by the symbol ``$\ll$" depends only on $A$.
	\end{lemma}
	\begin{proof}
		This is the Bombieri-Vinogradov theorem. 
	\end{proof}
	
	\begin{lemma}
		If $1\leq d<x$, $0\leq h<d$, $(h,d)=1$, then
		\begin{align*}
			\pi(x;d,h)=\sum_{\substack{p\leq x\\p\equiv h\mkern-15mu\pmod{d}}}1\leq \frac{(2+o(1))x}{\varphi(d)\log(x/d)}.
		\end{align*}
	\end{lemma}
	\begin{proof}
		This is the Brun-Titchmarsh inequity.
	\end{proof}
	
	\begin{lemma}
		There exist two functions $C_2(\theta)>C_1(\theta)>0$, defined on the intervals $(0,17/32)$, such that for each fixed integer $a\in\mathbb{Z}^*$, for each fixed real $A>0$, and all sufficiently large $Q=x^{\theta}$, the inequalities
		\begin{align*}
			C_1(\theta)\frac{\pi(x)}{\varphi(q)}\leq \pi(x;q,a)\leq C_2(\theta)\frac{\pi(x)}{\varphi(q)}
		\end{align*}
		hold for all primes $q\in (Q,2Q]$ with at most $O(Q(\log Q)^{-A})$ exceptions, where the implied constant depends only on $a$, $A$ and $\theta$. Moreover, for any fixed $\varepsilon>0$, these functions can be chosen to satisfy the following properties:
		\begin{itemize}
			\item $C_1(\theta)$ is monotonic decreasing, and $C_2(\theta)$ is monotonic increasing.
			\item $C_1(1/2)=1-\varepsilon$  and $C_2(1/2)=1+\varepsilon$.
		\end{itemize}
		\begin{proof}
			This is [1, Lemma 2.1].
		\end{proof}
		
	\end{lemma}

	\subsection*{Convention.} We use $\varepsilon$ to denote a sufficiently small positive number, and the value of $\varepsilon$ may change from statement to statement. We use $\mu(n)$ and $\varphi(n)$ to denote the M$\mathrm{\ddot{o}}$bius function and Euler's function, respectively. We use $\nu(n)$ to denote the number of distinct prime factors of $n$. By $(m_1,m_2)$ we denote the greatest common divisor of $m_1$ and $m_2$. We use the standard asymptotic notation $f\ll g,f\gg g,f=O(g),f=o_{x\rightarrow\infty}(g)$ from analytic number theory, and indicate that the implicit constants depend on some parameter $\varepsilon$ through subscripts. By $m\sim M$ we denote $M< m\leq2M$. Statements like $f(x)\ll x^{o(1)}g(x)$ should be read as $\forall \varepsilon>0$, $f(x)\ll_\varepsilon x^\varepsilon g(x)$. Define
	\begin{align*}
		P(x):=\prod_{p\leq x}p,\quad P(y,z):=\prod_{y<p\leq z}p.
	\end{align*}
	Define
	\begin{align*}
		H(n):=\prod_{p>2,p|n}\left(\frac{p-1}{p-2}\right).
	\end{align*}

	\section{Proof of Theorem 1.4 and Theorem 1.5}
	In this section, we mainly focus on the proof of Theorem 1.4.  And Theorem 1.5 is an easy corollary. We consider $P^+_y(p-1)<P^+_y(p+1)$.
	\subsection{Case 1: $3\leq y\leq  \log\log x$} In this case, we have
	\begin{align*}
		P(y)=\prod_{p\leq y}p=\exp\left(\sum_{p\leq y}\log p\right)\leq (\log x)^2.
	\end{align*}
	Define $2qp':=(p+1,P(y))$ with $P^+(q)<p'$, $p'\geq 3$. (\(p'\ge 3\) must exist; otherwise, \((p-1,P(y))=(p+1,P(y))=2\), which contradicts \( p^2\equiv 1\mkern-7mu\pmod{3}\).) We write $P^+_y(p-1)<P^+_y(p+1)$ as $(p-1,P(p',y))=1$.
	By the Siegel-Walfisz theorem, we have
	\begin{align*}
		\sum_{\substack{p\leq x\\(p+1,P(y))=2qp'\\(p-1,P(p',y))=1}}1&=\sum_{2qp'\sigma_1|P(y)}\mu(\sigma_1)\sum_{\substack{\sigma_2|P(p'   ,y)\\(\sigma_1,\sigma_2)=1}}\mu(\sigma_2)\sum_{\substack{p\leq x\\2qp'\sigma_1|(p+1)\\\sigma_2|(p-1)}}1\\&=\sum_{2qp'\sigma_1|P(y)}\mu(\sigma_1)\sum_{\substack{\sigma_2|P(p',y)\\(\sigma_1,\sigma_2)=1}}\mu(\sigma_2)\frac{\pi(x)}{\varphi(2qp'\sigma_1\sigma_2)}+O(x\exp(-c\sqrt{\log x})).\tag{3.1}
	\end{align*}
	For the summation, we have
	\begin{align*}
		\sum_{2qp'\sigma_1|P(y)}\mu(\sigma_1)\sum_{\substack{\sigma_2|P(p',y)\\(\sigma_1,\sigma_2)=1}}\mu(\sigma_2)&\frac{\pi(x)}{\varphi(2qp'\sigma_1\sigma_2)}=\frac{\pi(x)}{\varphi(qp')}\sum_{2qp'\sigma_1|P(y)}\frac{\mu(\sigma_1)}{\varphi(\sigma_1)}\prod_{\substack{p'<p\leq y\\(p,\sigma_1)=1}}\left(1-\frac{1}{p-1}\right)\\&=\frac{\pi(x)}{\varphi(qp')}\prod_{\substack{p'<p\leq y}}\left(\frac{p-2}{p-1}\right)\sum_{2qp'\sigma_1|P(y)}\frac{\mu(\sigma_1)}{\varphi(\sigma_1)}\prod_{\substack{p'<p\leq y\\p|\sigma_1}}\left(\frac{p-1}{p-2}\right)\\&=\frac{\pi(x)}{\varphi(qp')}\prod_{\substack{p'<p\leq y}}\left(\frac{p-2}{p-1}\right)\prod_{\substack{p'<p\leq y}}\left(\frac{p-3}{p-2}\right)\prod_{\substack{2<p<p'\\(p,q)=1}}\left(\frac{p-2}{p-1}\right)\\&=\frac{\pi(x)}{\varphi(qp')}\prod_{\substack{p'<p\leq y}}\left(\frac{p-3}{p-1}\right)\prod_{\substack{2<p<p'}}\left(\frac{p-2}{p-1}\right)\prod_{\substack{2<p<p'\\p|q}}\left(\frac{p-1}{p-2}\right).\tag{3.2}
	\end{align*}
	Summing over all $q$ and $p'$, we have
	\begin{align*}
		&\sum_{\substack{q,p'\\2qp'|P(y)\\P^+(q)<p'}}\frac{\pi(x)}{\varphi(qp')}\prod_{\substack{p'<p\leq y}}\left(\frac{p-3}{p-1}\right)\prod_{\substack{2<p<p'}}\left(\frac{p-2}{p-1}\right)\prod_{\substack{2<p<p'\\p|q}}\left(\frac{p-1}{p-2}\right)\\&=\pi(x)\sum_{2<p'\leq y}\frac{1}{p'-1}\prod_{\substack{p'<p\leq y}}\left(\frac{p-3}{p-1}\right)\prod_{\substack{2<p<p'}}\left(\frac{p-2}{p-1}\right)\sum_{2q|P(p'-1)}\frac{1}{\varphi(q)}\prod_{\substack{2<p<p'\\p|q}}\left(\frac{p-1}{p-2}\right)\\&=\pi(x)\sum_{2<p'\leq y}\frac{1}{p'-1}\prod_{\substack{p'<p\leq y}}\left(\frac{p-3}{p-1}\right)\prod_{\substack{2<p<p'}}\left(\frac{p-2}{p-1}\right)\prod_{2<p<p'}\left(\frac{p-1}{p-2}\right)\\&=\pi(x)\sum_{2<p'\leq y}\frac{1}{p'-1}\prod_{\substack{p'<p\leq y}}\left(\frac{p-3}{p-1}\right).\tag{3.3}
	\end{align*}
	Let \(\nu(n)\) denote the number of distinct prime factors of $n$. Noting that
	\begin{align*}
		\sum_{2<p'\leq y}\frac{2}{p'-1}\prod_{\substack{p'<p\leq y}}\left(\frac{p-3}{p-1}\right)&=\sum_{2<p'\leq y}\frac{2}{p'-1}\prod_{\substack{p'<p\leq y}}\left(1+\frac{-2}{p-1}\right)\\&=\sum_{2<p'\leq y}\frac{2}{p'-1}\sum_{r|P(p',y)}\frac{(-2)^{\nu(r)}}{\varphi(r)}\\&=-\sum_{2s|P(y),s>1}\frac{(-2)^{\nu(s)}}{\varphi(s)}\\&=-\prod_{2<p\leq y}\left(\frac{p-3}{p-1}\right)+1\\&=1,
	\end{align*}
	we get the required main term.
	For the error term, since both $q$ and \(p'\) are at most \(O(\log x)\) in number, this term is negligible. We complete the proof of this case.
	
	\subsection{Case 2: $\log\log x<y\ll x^{o(1)}$}
	We start from the following expression
	\begin{align*}
		\sum_{\substack{p\leq x\\ P^+_y(p-1)<P^+_y(p+1)}}1\geq \sum_{\substack{p\leq x\\ P^+_y(p+1)>y^{c}}}1-\sum_{\substack{p\leq x\\ y^{c}<P^+_y(p+1)<P^+_y(p-1)}}1=\mathscr{S}_A-\mathscr{S}_B,\tag{3.4}
	\end{align*}
	where $\log\log x<y\ll x^{o(1)}$, $0<c<1/100$. For $\mathscr{S}_A$, we have
	\begin{align*}
		\mathscr{S}_A&=\sum_{\substack{p\leq x\\(p+1,P(y^c,y))>1}}1\\&=\sum_{y^c<p'_1\leq y }\sum_{\substack{p\leq x\\p'_1|(p+1)}}1-\frac{1}{2!}\sum_{y^c<p'_1,p_2'\leq y }\sum_{\substack{p\leq x\\p'_1p_2'|(p+1)}}1+\frac{1}{3!}\sum_{y^c<p'_1,p_2',p_3'\leq y }\sum_{\substack{p\leq x\\p'_1p_2'|(p+1)}}1-\cdots+o(\pi(x))\\&\geq \sum_{1\leq j\leq J}\frac{(-1)^{j+1}}{j!}\sum_{y^c<p_1',\cdots,p_j'\leq y}\sum_{\substack{p\leq x\\p_1'\cdots p_j'|(p+1)}}1+o(\pi(x))\\&=\sum_{1\leq j\leq J}\frac{(-1)^{j+1}}{j!}\sum_{y^c<p_1',\cdots,p_j'\leq y}\frac{\pi(x)}{\varphi(p_1'\cdots p_j')}\\&\quad +\sum_{1\leq j\leq J}\frac{(-1)^{j+1}}{j!}\sum_{y^c<p_1',\cdots,p_j'\leq y}\Bigg(\sum_{\substack{p\leq x\\p_1'\cdots p_j'|(p+1)}}1-\frac{\pi(x)}{\varphi(p_1'\cdots p_j')}\Bigg)+o(\pi(x)),\tag{3.5}
	\end{align*}
	where $J=J(c,\varepsilon)$ is an even parameter to be determined. By the Bombieri-Vinogradov theorem, the second summation can be bounded by
	\begin{align*}
		\ll \sum_{q\leq y^J}\Bigg|\sum_{\substack{p\leq x\\p\equiv -1\mkern-15mu\pmod{q}}}1-\frac{\pi(x)}{\varphi(q)}\Bigg|\ll\frac{x}{(\log x)^A}.
	\end{align*}
	Now we begin to estimate the main term. If \(p_i = p_j\) occurs for $i\neq j$, the contribution of these terms is \(O(\pi(x)y^{-c})\). Thus we can write
	\begin{align*}
		\sum_{1\leq j\leq J}\frac{(-1)^{j+1}}{j!}\sum_{y^c<p_1',\cdots,p_j'\leq y}\frac{\pi(x)}{\varphi(p_1'\cdots p_j')}&=\pi(x)\sum_{1\leq j\leq J}\frac{(-1)^{j+1}}{j!}\left(\sum_{y^c<p_1'\leq y}\frac{1}{\varphi(p_1')}\right)^j+o\left(\pi(x)\right)\\&=\Bigg(\sum_{1\leq j\leq J}\frac{(-1)^{j+1}}{j!}\left(\log\frac{1}{c}\right)^j+o(1)\Bigg)\pi(x),\tag{3.6}
	\end{align*}
	since the prime number theorem implies
	\begin{align*}
		\sum_{p\leq x}\frac{1}{\varphi(p)}=\log\log x+C+O\left(\frac{1}{\log x}\right).
	\end{align*}
	Note that
	\begin{align*}
		e^{-x}=1-x+\frac{x^2}{2!}-\frac{x^3}{3!}+\cdots.
	\end{align*}
	Thus for sufficiently large $J$, we have
	\begin{align*}
		\sum_{1\leq j\leq J}\frac{(-1)^{j+1}}{j!}\left(\log\frac{1}{c}\right)^j\geq 1-\exp\left(-\log\frac{1}{c}\right)-c\varepsilon=1-c-c\varepsilon.\tag{3.7}
	\end{align*}
	Conclude from (3.5)--(3.7) that
	\begin{align*}
		\mathscr{S}_A\geq (1-c-c\varepsilon+o(1))\pi(x).\tag{3.8}
	\end{align*}
	
	For $\mathscr{S}_B$, we have
	\begin{align*}
		\mathscr{S}_B=&\sum_{\substack{p\leq x\\(p+1,P(y^c,P^+_y(p-1)))>1\\(p+1,P(P^+_y(p-1),y))=1}}1\\=&\sum_{y^c<p_1\leq y}\sum_{\substack{p\leq x\\p_1|(p-1)\\(p+1,P(y^c,p_1))>1\\(p+1,P(p_1,y))=1}}1-\sum_{y^c<p_1<p_2\leq y}\sum_{\substack{p\leq x\\p_1p_2|(p-1)\\(p+1,P(y^c,p_1))>1\\(p+1,P(p_1,y))=1}}1+\sum_{y^c<p_1<p_2<p_3\leq y}\sum_{\substack{p\leq x\\p_1p_2p_3|(p-1)\\(p+1,P(y^c,p_1))>1\\(p+1,P(p_1,y))=1}}1\\&-\sum_{y^c<p_1<p_2<p_3<p_4\leq y}\sum_{\substack{p\leq x\\p_1p_2p_3p_4|(p-1)\\(p+1,P(y^c,p_1))>1\\(p+1,P(p_1,y))=1}}1+\cdots+o(\pi(x))\\\leq &\sum_{1\leq j\leq J'}(-1)^{j+1}\sum_{y^c<p_1<\cdots<p_j\leq y}\sum_{\substack{p\leq x\\p_1\cdots p_j|(p-1)\\(p+1,P(y^c,p_1))>1\\(p+1,P(p_1,y))=1}}1+o(\pi(x)),\tag{3.9}
	\end{align*}
	where $J'=J'(c,\varepsilon)$ is an odd parameter to be determined. If $p_1\in(y^{1-c\varepsilon},y)$, the contribution of these terms is
	\begin{align*}
		\leq (2c\varepsilon+o(1))\pi(x).\tag{3.10}
	\end{align*}
	Note that for any $y^c<p_1\leq y^{1-c\varepsilon}$, we have
	\begin{align*}
		\sum_{\substack{p\leq x\\p_1\cdots p_j|(p-1)\\(p+1,P(y^c,p_1))>1\\(p+1,P(p_1,y))=1}}1&=\sum_{\substack{p\leq x\\p_1\cdots p_j|(p-1)\\(p+1,P(p_1,y))=1}}1-\sum_{\substack{p\leq x\\p_1\cdots p_j|(p-1)\\(p+1,P(y^c,y))=1}}1\\&=\sum_{\substack{p\leq x\\p_1\cdots p_j|(p-1)\\(p+1,P(y^c,y))>1}}1-\sum_{\substack{p\leq x\\p_1\cdots p_j|(p-1)\\(p+1,P(p_1,y))>1}}1,
	\end{align*}
	since
	\begin{align*}
		\sum_{\substack{p\leq x\\p_1\cdots p_j|(p-1)\\(p+1,P(p_1,y))=1}}1=\sum_{\substack{p\leq x\\p_1\cdots p_j|(p-1)}}1-\sum_{\substack{p\leq x\\p_1\cdots p_j|(p-1)\\(p+1,P(p_1,y))>1}}1.
	\end{align*}
	Similar to the estimate of $\mathscr{S}_A$ (take $J$ to be odd, and the inequality sign reverses from \(\geq\) to \(\leq\)), we have
	\begin{align*}
		&\sum_{\substack{y^c<p_1\leq y^{1-c\varepsilon}\\y^c<p_1<\cdots<p_j\leq y}}\sum_{\substack{p\leq x\\p_1\cdots p_j|(p-1)\\(p+1,P(y^c,y))>1}}1\leq \sum_{\substack{y^c<p_1\leq y^{1-c\varepsilon}\\y^c<p_1<\cdots<p_j\leq y}}(1-c+c^2\varepsilon+o(1))\frac{\pi(x)}{\varphi(p_1\cdots p_j)},\\&\sum_{\substack{y^c<p_1\leq y^{1-c\varepsilon}\\y^c<p_1<\cdots<p_j\leq y}}\sum_{\substack{p\leq x\\p_1\cdots p_j|(p-1)\\(p+1,P(p_1,y))>1}}1\geq \sum_{\substack{y^c<p_1\leq y^{1-c\varepsilon}\\y^c<p_1<\cdots<p_j\leq y}}\left(1-\frac{\log p_1}{\log y}-c^2\varepsilon+o(1)\right)\frac{\pi(x)}{\varphi(p_1\cdots p_j)},
	\end{align*}
	and
	\begin{align*}
		&\sum_{\substack{y^c<p_1\leq y^{1-c\varepsilon}\\y^c<p_1<\cdots<p_j\leq y}}\sum_{\substack{p\leq x\\p_1\cdots p_j|(p-1)\\(p+1,P(y^c,p_1))>1\\(p+1,P(p_1,y))=1}}1\leq\sum_{\substack{y^c<p_1\leq y^{1-c\varepsilon}\\y^c<p_1<\cdots<p_j\leq y}} \left(\frac{\log p_1}{\log y}-c+2c^2\varepsilon+o(1)\right)\frac{\pi(x)}{\varphi(p_1\cdots p_j)},\\&\sum_{\substack{y^c<p_1\leq y^{1-c\varepsilon}\\y^c<p_1<\cdots<p_j\leq y}}\sum_{\substack{p\leq x\\p_1\cdots p_j|(p-1)\\(p+1,P(y^c,p_1))>1\\(p+1,P(p_1,y))=1}}1\geq\sum_{\substack{y^c<p_1\leq y^{1-c\varepsilon}\\y^c<p_1<\cdots<p_j\leq y}} \left(\frac{\log p_1}{\log y}-c-2c^2\varepsilon+o(1)\right)\frac{\pi(x)}{\varphi(p_1\cdots p_j)}.
	\end{align*}
	Now we can estimate $\mathscr{S}_B$ by
	\begin{align*}
		\mathscr{S}_B&\leq \sum_{1\leq j\leq J'}(-1)^{j+1}\sum_{\substack{y^c<p_1\leq y^{1-c\varepsilon}\\y^c<p_1<\cdots<p_j\leq y}}\left(\frac{\log p_1}{\log y}-c+(-1)^{j+1}2c^2\varepsilon+o(1)\right)\frac{\pi(x)}{\varphi(p_1\cdots p_j)}\\&\leq \sum_{1\leq j\leq J'}(-1)^{j+1}\sum_{\substack{y^c<p_1\leq y^{1-c\varepsilon}\\y^c<p_1<\cdots<p_j\leq y}}\left(\frac{\log p_1}{\log y}-c\right)\frac{\pi(x)}{\varphi(p_1\cdots p_j)}+(2c\varepsilon+o(1))\pi(x)\\&=\pi(x)\sum_{y^c<p_1\leq y^{1-c\varepsilon}}\frac{1}{p_1}\left(\frac{\log p_1}{\log y}-c\right)\sum_{1\leq j\leq J'}\frac{(-1)^{j+1}}{(j-1)!}\left(\log \frac{\log y}{\log p_1}\right)^{j-1}+(2c\varepsilon+o(1))\pi(x).\tag{3.11}
	\end{align*}
	Then for sufficiently large $J'$, by partial summation, we have
	\begin{align*}
		&\pi(x)\sum_{y^c<p_1\leq y^{1-c\varepsilon}}\frac{1}{p_1}\left(\frac{\log p_1}{\log y}-c\right)\sum_{j\leq J'}\frac{(-1)^{j+1}}{(j-1)!}\left(\log \frac{\log y}{\log p_1}\right)^{j-1}\\&\leq \pi(x)\sum_{y^c<p_1\leq y^{1-c\varepsilon}}\frac{1}{p_1}\left(\frac{\log p_1}{\log y}-c\right)\frac{\log p_1}{\log y}+(c\varepsilon+o(1))\pi(x)\\&=\pi(x)\int_{y^c}^{y^{1-c\varepsilon}}\frac{1}{t\log t}\left(\frac{\log t}{\log y}-c\right)\frac{\log t}{\log y}\mathrm{d}t+(c\varepsilon+o(1))\pi(x)\\&\leq\Bigg(\frac{1}{2}-\frac{c^2}{2}-c(1-c)+2 c\varepsilon+o(1)\Bigg)\pi(x).\tag{3.12}
	\end{align*}
	
	Conclude from (3.4), (3.8), (3.9)--(3.12) that
	\begin{align*}
		\mathscr{S}_B\leq \Bigg(\frac{1}{2}-\frac{c^2}{2}-c(1-c)+4 c\varepsilon+o(1)\Bigg)\pi(x)
	\end{align*}
	and 
	\begin{align*}
		\sum_{\substack{p\leq x\\ P^+_y(p-1)<P^+_y(p+1)}}1\geq (g(c)+o(1))x,
	\end{align*}
	where $\log\log x<y\ll x^{o(1)}$, $0<c<1/100$,
	\begin{align*}
		g(c)=\frac{1}{2}-\frac{c^2}2-5c\varepsilon.
	\end{align*}
	In the same time, we have $\lim_{c\rightarrow 0^+}g(c)=\frac{1}{2}$
	and
	\begin{align*}
		\sum_{\substack{p\leq x\\ P^+_y(p-1)<P^+_y(p+1)}}1\geq \left(\frac{1}{2}-\varepsilon+o(1)\right)x.\tag{3.13}
	\end{align*}
	Similarly, we also have
	\begin{align*}
		\sum_{\substack{p\leq x\\ P^+_y(p-1)>P^+_y(p+1)}}1\geq \left(\frac{1}{2}-\varepsilon+o(1)\right)x,
	\end{align*}
	i.e.,
	\begin{align*}
		\sum_{\substack{p\leq x\\ P^+_y(p-1)<P^+_y(p+1)}}1\leq \left(\frac{1}{2}+\varepsilon+o(1)\right)x.\tag{3.14}
	\end{align*}
	By (3.13) and (3.14), as \(\varepsilon\to 0\), we complete the proof for the second case.
	
	Combining the two cases, for $3\leq y=y(x)\ll x^{o(1)}$, we have
	\begin{align*}
		\lim_{x\rightarrow\infty}\frac{1}{\pi(x)}\sum_{\substack{p\leq x\\ P^+_y(p-1)<P^+_y(p+1)}}1=\frac{1}{2}.
	\end{align*}
	We complete the proof of Theorem 1.4. The proof of Theorem 1.5 is similar (and simpler, since the Siegel-Walfisz theorem and the Bombieri-Vinogradov theorem are not required).
	
	\begin{remark}
		In this section, the method we employ should also be applicable to the following form
		\begin{align*}
			\lim_{x\rightarrow\infty}\frac{1}{\pi(x)}\sum_{\substack{p\leq x\\P^+_y(p+a_1)<P^+_y(p+a_2)<\cdots<P^+_y(p+a_j)}}1=\frac{1}{j!},
		\end{align*}
		for $j\geq 3$, $O(1)\ll y=x^{o(1)}$, $1\leq |a_1|,|a_2|,\cdots,|a_j|\ll O(1)$, $a_{i_1}\neq a_{i_2}$ if $i_1\neq i_2$.
	\end{remark}
	\section{Proof of Theorem 1.6}
	In the proof of the second case of Theorem 1.4, we only require $y$ to be small with no additional constraints. In this setting, we may iterate $J$ (or \(J'\)) times. In order to apply the Bombieri-Vinogradov theorem, it therefore suffices to take \(y < x^{\min(1/J,1/J')/3}\). Then for $0<\alpha<\min(1/J,1/J')/3$ and $y=x^\alpha$, there exists $h_0(\alpha)>1/2-\varepsilon$ such that
	\begin{align*}
		\sum_{\substack{p\leq x\\ P^+_y(p-1)<P^+_y(p+1)}}1\geq \left(h_0(\alpha)+o(1)\right)\pi(x).
	\end{align*} 
	In particular, the function $h_0$ satisfies $\lim_{\alpha\rightarrow 0^+}h_0(\alpha)=1/2$.
	
	In the following, let $y=x^{\alpha}$.
	\subsection{Case 1: $0<\alpha< 1/2-\delta$}
	We start from the following expression
	\begin{align*}
		\sum_{\substack{p\leq x\\ P^+_y(p-1)<P^+_y(p+1)}}1\geq \sum_{\substack{p\leq x\\ P^+_y(p+1)>x^{c}}}1-\sum_{\substack{p\leq x\\ x^{c}<P^+_y(p+1)<P^+_y(p-1)}}1=\mathscr{S}_A-\mathscr{S}_B,\tag{4.1}
	\end{align*}
	where $0<c<\alpha$ and $0<\alpha<1/2-\varepsilon$.
	For $\mathscr{S}_A$, we have
	\begin{align*}
		\mathscr{S}_A&\geq\sum_{\substack{x^c<p_1\leq x^\alpha}}\sum_{\substack{p\leq x\\p+1\equiv 0\mkern-15mu\pmod{p_1}}}1-\sum_{\substack{x^c<p_1< p_2\leq x^\alpha}}\sum_{\substack{p\leq x\\p+1\equiv 0\mkern-15mu\pmod{p_1p_2}}}1\\&=\sum_{\substack{x^c<p_1\leq x^\alpha}}\frac{\pi(x)}{\varphi(p_1)}+\sum_{\substack{x^c<p_1\leq x^\alpha}}\Bigg(\sum_{\substack{p\leq x\\p+1\equiv 0\mkern-15mu\pmod{p_1}}}1-\frac{\pi(x)}{\varphi(p_1)}\Bigg)\\&\quad-\sum_{\substack{x^c<p_1< p_2\leq x^\alpha}}\sum_{\substack{p\leq x\\p+1\equiv 0\mkern-15mu\pmod{p_1p_2}}}1.
	\end{align*}
	By the Bombieri-Vinogradov theorem, the second summation can be bounded by
	\begin{align*}
		\ll \frac{x}{(\log x)^A}.
	\end{align*}
	Then by the prime number theorem, $\mathscr{S}_A$ becomes
	\begin{align*}
		\mathscr{S}_A&\geq\sum_{\substack{x^c<p_1\leq x^\alpha}}\frac{\pi(x)}{\varphi(p_1)}-\sum_{\substack{x^c<p_1<p_2\leq x^\alpha}}\sum_{\substack{p\leq x\\p+1\equiv 0\mkern-15mu\pmod{p_1p_2}}}1+O\left(\frac{x}{(\log x)^A}\right)\\&=\left(\log\frac{\alpha}{c}+o(1)\right)\pi(x)-\mathscr{R}_1,\tag{4.2}
	\end{align*}
	where 
	\begin{align*}
		\mathscr{R}_1=\sum_{\substack{x^c<p_1< p_2\leq x^\alpha}}\sum_{\substack{p\leq x\\p+1\equiv 0\mkern-15mu\pmod{p_1p_2}}}1.
	\end{align*}
	Now we begin to estimate $\mathscr{R}_1$ and $\mathscr{S}_B$. For $\mathscr{S}_B$, we have
	\begin{align*}
		\mathscr{S}_B\leq \sum_{\substack{x^c<p_1< p_2\leq x^\alpha}}\sum_{\substack{p\leq x\\p+1\equiv 0\mkern-15mu\pmod{p_1}\\p-1\equiv 0\mkern-15mu\pmod{p_2}}}1
	\end{align*}
	By the Brun-Titchmarsh inequity, we have
	\begin{align*}
		\sum_{\substack{p\leq x\\p\equiv a\mkern-15mu\pmod{q}}}1\leq \frac{(2+o(1))x}{\varphi(q)\log(x/q)},
	\end{align*}
	then we have
	\begin{align*}
		\mathscr{R}_1+\mathscr{S}_B&\leq \left(\frac{4}{1-2\alpha}+o(1)\right)\pi(x)\sum_{\substack{x^c<p_1< p_2\leq x^\alpha}}\frac{1}{(p_1-1)(p_2-1)}\\&=\Bigg(\frac{2}{1-2\alpha}\left(\log\frac{\alpha}{c}\right)^2+o(1)\Bigg)\pi(x)\tag{4.3}
	\end{align*}
	Finally, for $0<c<\alpha<1/2-\varepsilon$, we have
	\begin{align*}
		\sum_{\substack{p\leq x\\ P^+_y(p-1)<P^+_y(p+1)}}1\geq (g_1(c,\alpha)+o(1))\pi(x),
	\end{align*}
	where
	\begin{align*}
		g_1(c,\alpha)=\log\frac{\alpha}{c}-\frac{2}{1-2\alpha}\left(\log\frac{\alpha}{c}\right)^2.
	\end{align*}
	For any $0<\alpha<1/2-\varepsilon$, there exists $h_1(\alpha)>0$ such that 
	\begin{align*}
		\sum_{\substack{p\leq x\\ P^+_y(p-1)<P^+_y(p+1)}}1\geq (h_1(\alpha)+o(1))\pi(x),\tag{4.5}
	\end{align*}
	where
	\begin{align*}
		h_1(\alpha)=\sup_{0<c<\alpha}g(c,\alpha)\geq \Bigg(\log\frac{\alpha}{c}-\frac{2}{1-2\alpha}\left(\log\frac{\alpha}{c}\right)^2\Bigg)\Bigg|_{\log\frac{\alpha}{c}=\frac{1-2\alpha}{4}}=\frac{1-2\alpha}{8}>0.
	\end{align*}
	
	\subsection{Case 2: $1/2-\delta<\alpha\leq 1/2$}
	Within this section, a minor changes to the expression in Case 1 are required. We start from the following expression
	\begin{align*}
		\sum_{\substack{p\leq x\\ P^+_y(p-1)<P^+_y(p+1)}}1&\geq \sum_{\substack{p\leq x\\ P^+_y(p+1)>x^{c}}}1-\sum_{\substack{p\leq x\\ x^{c}<P^+_y(p+1)<P^+_y(p-1)}}1\\&\geq \sum_{\substack{p\leq x\\ P^+_y(p+1)>x^{c}}}1-\sum_{\substack{p\leq x\\ x^{c}<P^+_y(p+1)<P^+_y(p-1)<x^{c_1}}}1-\sum_{\substack{P^+_y(p-1)\geq x^{c_1}}}1\\&=\mathscr{S}_A-\mathscr{S}_B-\mathscr{S}_C.\tag{4.6}
	\end{align*}
	where $0<c<c_1<\alpha$ and $1/4<\alpha\leq 1/2$.
	Similar to Case 1, we have
	\begin{align*}
		\mathscr{S}_A\geq&\left(\log\frac{\alpha}{c}+o(1)\right)\pi(x)-\mathscr{R}_2,\\\mathscr{S}_B\leq& \sum_{\substack{x^c<p_1< p_2\leq x^{c_1}}}\sum_{\substack{p\leq x\\p+1\equiv 0\mkern-15mu\pmod{p_1}\\p-1\equiv 0\mkern-15mu\pmod{p_2}}}1\leq \left(\frac{1}{1-2c_1}\left(\log\frac{c_1}{c}\right)^2+o(1)\right)\pi(x),\\\mathscr{S}_C\leq&\sum_{x^{c_1}\leq p_1\leq x^{\alpha}}\sum_{\substack{p\leq x\\p-1\equiv 0\mkern-15mu\pmod{p_1}}}1=\left(\log\frac{\alpha}{c_1}+o(1)\right)\pi(x),
	\end{align*}
	where
	\begin{align*}
		\mathscr{R}_2=\sum_{x^c<p_1, p_2\leq x^{\alpha}}\sum_{\substack{p\leq x\\p+1\equiv 0\mkern-15mu\pmod{p_1p_2}}}1.
	\end{align*}
	We only need to estimate $\mathscr{R}_2$.
	Without loss of generality, consider $x(\log x)^{-B}<p\leq x$. We write $p+1=lp_1p_2$, where $x^c<p_1\leq x^{\alpha}$, $x^{1-\alpha}(\log x)^{-B}<lp_1\leq x^{1-c}$ and $2|l$. 
	Define
	\begin{align*}
		\mathscr{A}_1(lp_1):=\#\{1\leq m\leq (x+1)/(lp_1):(-1+lp_1m)\ \text{and}\ m\ \text{are primes} \}.
	\end{align*}
	By Lemma 2.1, we have
	\begin{align*}
		\mathscr{A}_1(lp_1)\leq (8+o(1))\prod_p\left(1-\frac{\omega(p)-1}{p-1}\right)\left(1-\frac{1}{p}\right)^{-1}\frac{x}{lp_1\log^2(x/(lp_1))},
	\end{align*}
	where $\omega(p)$ denote the number of solutions in $m$ modulo $p$ of 
	\begin{align*}
		(-1+lp_1m)m\equiv 0\mkern-15mu\pmod{p}.
	\end{align*}
	We have $\omega(p)=1$ if $p|lp_1$, and $\omega(p)=2$ if $(lp_1,p)=1$. Then
	\begin{align*}
		\prod_p\left(1-\frac{\omega(p)-1}{p-1}\right)\left(1-\frac{1}{p}\right)^{-1}=2A_0\prod_{p>2,p|p_1l}\left(\frac{p-1}{p-2}\right).
	\end{align*}
	Now we can bound $\mathscr{R}_2$ by
	\begin{align*}
		\mathscr{R}_2&\leq\sum_{\substack{p_1,l\\x^c<p_1\leq x^{\alpha}\\x^{1-\alpha}(\log x)^{-B}<lp_1\leq x^{1-c}\\2|l}}\mathscr{A}_1(lp_1)+o(\pi(x))\\&\leq (16A_0+o(1))\frac{x}{((c-\varepsilon)\log x)^2}\sum_{\substack{p_1,l\\x^c<p_1\leq x^{\alpha}\\x^{1-\alpha}(\log x)^{-B}<lp_1\leq x^{1-c}\\2|l}}\frac{1}{lp_1}\prod_{p>2,p|l}\left(\frac{p-1}{p-2}\right).\tag{4.7}
	\end{align*}
	Define a multiplicative function $H(n)$ by
	\begin{align*}
		H(n):=\prod_{p>2,p|n}\left(\frac{p-1}{p-2}\right).
	\end{align*}
	The partial sum of $H(d)$ is well studied by La Bret\`{e}che, Pomerance and Tenenbaum [2] by
	employing the Selberg-Delange method: for any nonnegative integer $v$,
	\begin{align*}
		\sum_{\substack{d\leq y\\(d,a)=1}}H(2^vd)\sim \frac{A_1y}{2^{\varepsilon(a)}H(a)}\qquad(y\rightarrow\infty),
	\end{align*}
	where $\varepsilon(a)=1$ if $2|a$ and $\varepsilon(a)=0$ if $2\nmid a$, and $A_1$ is an absolute constant defined by
	\begin{align*}
		A_1:=\prod_{p>2}\left(1+\frac{1}{p(p-2)}\right)=\frac{1}{A_0}.
	\end{align*}
	Thus by partial summation, we have \begin{align*}
		\sum_{\substack{x^{1-\alpha}(\log x)^{-B}p_1^{-1}<l\leq x^{1-c}p_1^{-1}\\2|l}}\frac{H(l)}{l}&=	\sum_{\substack{x^{1-\alpha}(\log x)^{-B}p_1^{-1}2^{-1}<l\leq x^{1-c}p_1^{-1}2^{-1}}}\frac{H(l)}{2l}\\&=\left(\frac{A_1}{2}+o(1)\right)\int_{x^{1-\alpha}(\log x)^{-B}p_1^{-1}2^{-1}}^{x^{1-c}p_1^{-1}2^{-1}}\frac{1}{a}\mathrm{d}a\\&=\left(\frac{A_1}{2}+o(1)\right)(\alpha-c)\log x.
	\end{align*}
	Then (4.7) becomes
	\begin{align*}
		\mathscr{R}_2\leq \left(\frac{8(\alpha-c)}{(c-\varepsilon)^2}\log\frac{\alpha}{c}+o(1)\right)\pi(x).\tag{4.8}
	\end{align*}

	For $0<c<c_1<\alpha$, we thus have
	\begin{align*}
		\sum_{\substack{p\leq x\\ P^+_y(p-1)<P^+_y(p+1)}}1\geq (g_2(c,c_1,\alpha)+o(1))\pi(x),
	\end{align*}
	where
	\begin{align*}
		g_2(c,c_1,\alpha)&=\log\frac{\alpha}{c}-\frac{8(\alpha-c)}{(c-\varepsilon)^2}\log\frac{\alpha}{c}-\frac{1}{1-2c_1}\left(\log\frac{c_1}{c}\right)^2-\log\frac{\alpha}{c_1}\\&=\log\frac{c_1}{c}-\frac{8(\alpha-c)}{(c-\varepsilon)^2}\log\frac{\alpha}{c}-\frac{1}{1-2c_1}\left(\log\frac{c_1}{c}\right)^2\\&=\log\frac{c_1}{c}\Bigg(1-\frac{8(\alpha-c)}{(c-\varepsilon)^2}\frac{\log(\alpha/c)}{\log (c_1/c)}-\frac{1}{1-2c_1}\log\frac{c_1}{c}\Bigg)
	\end{align*}
	Suppose
	\begin{align*}
		0<\frac{\alpha-c}{c}<1,\quad \frac{1}{4}<c<c_1,\quad \frac{1}{2}-c_1<3(c_1-c).\tag{4.9}
	\end{align*}
	For $0<s<1$, we have $s/2<\log (1+s)\leq s$.
	Then we have
	\begin{align*}
		\frac{8(\alpha-c)}{(c-\varepsilon)^2}\frac{\log(\alpha/c)}{\log (c_1/c)}&\leq 	\frac{9(\alpha-c)}{c^2}\frac{\log(\alpha/c)}{\log (c_1/c)}\leq 72\frac{(\alpha-c)^2}{c(c_1-c)}\\&\leq 144\frac{(\alpha-c)^2}{(c_1-c)^2}\log\frac{c_1}{c}\leq 2304\log\frac{c_1}{c}
	\end{align*}
	since
	\begin{align*}
		\frac{(\alpha-c)^2}{(c_1-c)^2}=\left(1+\frac{\alpha-c_1}{c_1-c}\right)^2\leq 16
	\end{align*}
	Fix $c_1$. Taking
	\begin{align*}
		\log \frac{c_1}{c}=\frac{1}{2(2304+1/(1-2c_1))}\quad \text{or}\quad c=c_1\exp\left(-\frac{1}{2(2304+1/(1-2c_1))}\right),
	\end{align*}
	we have
	\begin{align*}
		\sup_{c\ \text{satisfies} \ (4.9)}g_2(c,c_1,\alpha)\geq \frac{1}{4(2304+1/(1-2c_1))}>0.\tag{4.10}
	\end{align*}
	For $\alpha=1/2$, there exists some \(c_1\in(0.45,1/2)\) satisfying
	\begin{align*}
		f(c_1,\alpha):=4c_1-3c_1\exp\left(-\frac{1}{2(2304+1/(1-2c_1))}\right)-\frac{1}{2}>0,
	\end{align*}
	since (we take limits in the relevant expressions)
	\begin{align*}
		f(\alpha,\alpha)=0,\quad \frac{\partial}{\partial c_1}f(c_1,\alpha)\Big|_{c_1=\alpha}&=-\frac{1}{2}.
	\end{align*}
	Thus we can find \(c_1\in(0.45,1/2)\) satisfying \(1/2-c_1<3(c_1-c)\). All conditions in (4.9) can be satisfied.
	Then the above expression (4.10) holds for all \(\alpha \in(c_1,1/2]\).
	
	The above may be summarised as follows: We can find $\delta>0$ such that for any $1/2-\delta\leq \alpha\leq 1/2$, there exists $j_1(\alpha)>0$ satisfying
	\begin{align*}
		\inf_{1/2-\delta\leq \alpha\leq 1/2}j_1(\alpha)>0,\quad\sum_{\substack{p\leq x\\P^+_y(p-1)<P^+_y(p+1)}}1\geq (j_1(\alpha)+o(1))\pi(x).\tag{4.11}
	\end{align*}

	\subsection{Case 3: $1/2\leq\alpha< 1/2+\delta'$}
	We start from the following expression
	\begin{align*}
		\sum_{\substack{p\leq x\\ P^+_y(p-1)<P^+_y(p+1)}}1&\geq \sum_{\substack{p\leq x\\ P^+_y(p+1)>x^{c}}}1-\sum_{\substack{p\leq x\\ x^{c}<P^+_y(p+1)<P^+_y(p-1)}}1\\&\geq \sum_{\substack{p\leq x\\ P^+_y(p+1)>x^{c}}}1-\sum_{\substack{p\leq x\\ x^{c}<P^+_y(p+1)<P^+_y(p-1)<x^{c_1}}}1-\sum_{\substack{P^+_y(p-1)\geq x^{c_1}}}1\\&\geq \sum_{\substack{p\leq x\\ (p+1,P(x^{c},x^{1/2}))>1}}1-\sum_{x^c<p_1<p_2<x^{c_1}}\sum_{\substack{p\leq x\\p+1\equiv 0\mkern-15mu\pmod{p_1}\\p-1\equiv 0\mkern-15mu\pmod{p_2}}}1\\&\quad-\sum_{x^{c_1}<p_1\leq x^{1/2}}\sum_{\substack{p\leq x\\p-1\equiv 0\mkern-15mu\pmod{p_1}}}1-\sum_{x^{1/2}<p_1\leq x^{\alpha}}\sum_{\substack{p\leq x\\p-1\equiv 0\mkern-15mu\pmod{p_1}}}1.
	\end{align*}
	where $0<c<c_1<1/2\leq\alpha< 17/32$.
	From the arguments in Case 2, we obtain that there exist \(c, c_1\) satisfying the requirements such that
	\begin{align*}
		&\sum_{\substack{p\leq x\\ (p+1,P(x^{c},x^{1/2}))>1}}1-\sum_{x^c<p_1<p_2<x^{c_1}}\sum_{\substack{p\leq x\\p+1\equiv 0\mkern-15mu\pmod{p_1}\\p-1\equiv 0\mkern-15mu\pmod{p_2}}}1-\sum_{x^{c_1}<p_1\leq x^{1/2}}\sum_{\substack{p\leq x\\p-1\equiv 0\mkern-15mu\pmod{p_1}}}1\\&\geq (j_1(1/2)+o(1))\pi(x),
	\end{align*}
	where $j_1(1/2)>0$ is given in (4.11). Now it remains to estimate the final summation. By the Brun-Titchmarsh inequity, we have
	\begin{align*}
		\sum_{x^{1/2}<p_1\leq x^{\alpha}}\sum_{\substack{p\leq x\\p-1\equiv 0\mkern-15mu\pmod{p_1}}}1\leq \left(\frac{2}{1-\alpha}+o(1)\right)\sum_{x^{1/2}<p_1\leq x^{\alpha}}\frac{\pi(x)}{p_1}=\left(\frac{2\log(2\alpha)}{1-\alpha}+o(1)\right)\pi(x),
	\end{align*}
	Thus we have
	\begin{align*}
		\sum_{\substack{p\leq x\\ P^+_y(p-1)<P^+_y(p+1)}}1\geq\left(j_2(\alpha)+o(1)\right)\pi(x),\tag{4.12}
	\end{align*}
	where $1/2\leq\alpha< 1/2+\delta'$ with $\delta'=j_1(1/2)/24$,
	\begin{align*}
		j_2(\alpha)=j_1(1/2)-\frac{2\log(2\alpha)}{1-\alpha}\geq j_1(1/2)-6(2\alpha-1)>j_1(1/2)/2>0.
	\end{align*}
	
	\subsection{Case 4: $1/2+\delta'<\alpha< 17/32$}
	We start from the following expression
	\begin{align*}
		\sum_{\substack{p\leq x\\ P^+_y(p-1)<P^+_y(p+1)}}1\geq \sum_{\substack{p\leq x\\ P^+_y(p+1)>x^{c}}}1-\sum_{\substack{p\leq x\\ x^{c}<P^+_y(p+1)<P^+_y(p-1)}}1=\mathscr{S}_A-\mathscr{S}_B,\tag{4.13}
	\end{align*}
	where $1/2+\varepsilon<c<\alpha$, $1/2+\delta'<\alpha<17/32$, for the $\delta'$ in Case 3.
	For $\mathscr{S}_A$, by Lemma 2.5, we have
	\begin{align*}
		\mathscr{S}_A=\sum_{x^c<p_1\leq x^\alpha}\sum_{\substack{p\leq x\\p+1\equiv 0\mkern-15mu\pmod{p_1}}}1&\geq (C(\alpha)+o(1))\sum_{x^c<p_1\leq x^\alpha}\frac{\pi(x)}{p_1-1}\\&=\left(C(\alpha)\log\frac{\alpha}{c}+o(1)\right)\pi(x),\tag{4.14}
	\end{align*}
	where $C(\alpha)=C_1(\alpha)$ for the $C_1(\cdot)$ in Lemma 2.5.
	
	Now we begin to estimate $\mathscr{S}_B$. Consider $x(\log x)^{-B}<p\leq x$. We write $p+1=2ap_1$ and $p-1=2bp_2$ with $x^{1-\alpha}(\log x)^{-B}/2<b<a<x^{1-c}/2$, $6|ab$ and $(a,b)=1$. 
	According to the Chinese Remainder Theorem, this system $2a|(n+1)$ and $2b|(n-1)$ has a unique solution $n_0$ with $0<n_0<2ab$.
	All the solutions of this system are given by
	\begin{align*}
		n_0+2abm,\quad m\in\mathbb{Z}.
	\end{align*}
	Define
	\begin{align*}
		\mathscr{A}_2(ab):=\#\{&0\leq m\leq (x-n_0)/(2ab):(n_0+1)/(2a)+bm,\\& n_0+2abm\ \text{and}\ (n_0-1)/(2b)+am\ \text{are primes} \}.
	\end{align*}
	By Lemma 2.1, we have
	\begin{align*}
		\mathscr{A}_2(ab)\leq (24+o(1))\prod_{p}\left(1-\frac{\omega_1(p)-1}{p-1}\right)\left(1-\frac{1}{p}\right)^{-2}\frac{x}{ab\log^3(x/(2ab))},
	\end{align*}
	where $\omega_1(p)$ denote the number of solutions in $m$ modulo $p$ of 
	\begin{align*}
		\left(\frac{n_0+1}{2a}+bm\right)(n_0+2abm)\left(\frac{n_0-1}{2b}+am\right)\equiv 0\mkern-15mu\pmod{p}.
	\end{align*}
	We have $\omega_1(p)=1$ if $p|ab$, and $\omega(p)=3$ if $(ab,p)=1$. Then
	\begin{align*}
		\prod_{p}\left(1-\frac{\omega_1(p)-1}{p-1}\right)\left(1-\frac{1}{p}\right)^{-2}=9A_0'\prod_{p>3,p|ab}\left(\frac{p-1}{p-3}\right),
	\end{align*}
	where
	\begin{align*}
		A_0'=\prod_{p>3}\frac{p^2(p-3)}{(p-1)^3}
	\end{align*}
	Thus $\mathscr{S}_B$ becomes\begin{align*}
		\mathscr{S}_B&\leq \sum_{\substack{x^{1-\alpha}(\log x)^{-B}/2<b<a<x^{1-c}/2\\6|ab, (a,b)=1}}\mathscr{A}_2(ab)+o(\pi(x))\\&\leq (216A_0'+o(1))\frac{x}{((2c-1-\varepsilon)\log x)^3}\sum_{\substack{x^{1-\alpha}(\log x)^{-B}/2<b<a<x^{1-c}/2\\6|ab, (a,b)=1}}\frac{1}{ab}\prod_{p>3,p|ab}\left(\frac{p-1}{p-3}\right).\tag{4.16}
	\end{align*}
	We can rewrite $6|ab$ as $6|a$ or $6|b$ or $2|a$ and $3|b$ or $3|a$ and $2|b$.
	Define a multiplicative function $H_1(n)$ by
	\begin{align*}
		H_1(n):=\prod_{p>3,p|n}\left(\frac{p-1}{p-3}\right).
	\end{align*}
	Note that
	\begin{align*}
		\sum_{\substack{n\geq 1\\(n,a)=1}}\frac{H_1(n)}{n^s}=\zeta(s)G_a(s)\left(1-\frac{\varepsilon_1(a)}{2^s}\right)\left(1-\frac{\varepsilon_2(a)}{3^s}\right)\prod_{\substack{p|a,p>3}}\left(1-\frac{1}{p^s}\right)
	\end{align*}
	where $\varepsilon_1(a)=1_{2|a}$ and $\varepsilon_2(a)=1_{3|a}$,
	\begin{align*}
		G_a(s)=\prod_{\substack{p>3\\(p,a)=1}}\left(1+\frac{2}{p^s(p-3)}\right).
	\end{align*}
	The functions $G_a$ can be analytically
	continued to the half-plane $Re\ s > 0$. By a standard treatment,
	for $y\gg a$ and $y\rightarrow \infty$, we have
	\begin{align*}
		\sum_{\substack{d\leq y\\(d,a)=1}}H_1(d)=(1+o(1))\frac{A_1'y}{2^{\varepsilon_1(a)}(3/2)^{\varepsilon_2(a)}H_1'(a)},
	\end{align*}
	where 
	\begin{align*}
		A_1'=\prod_{p>3}\frac{(p-1)(p-2)}{p(p-3)},\quad H_1'(a)=\prod_{p>3,p|a}\left(\frac{p-2}{p-3}\right).
	\end{align*}
	Then we have (also see \S 4.2)
	\begin{align*}
		&\sum_{\substack{x^{1-\alpha}(\log x)^{-B}/2<b<a<x^{1-c}/2\\6|a, (a,b)=1}}\frac{H_1(ab)}{ab}\\&=\sum_{\substack{x^{1-\alpha}(\log x)^{-B}/2<b<x^{1-c}/2\\(b,6)=1}}\frac{H_1(b)}{6b}\sum_{\substack{b/6<a<x^{1-c}/12\\(a,b)=1}}\frac{H_1(a)}{a}\\&=\left(\frac{A'_1}{6}+o(1)\right)\sum_{\substack{x^{1-\alpha}(\log x)^{-B}/2<b<x^{1-c}/2\\(b,6)=1}}\frac{H_1(b)}{bH_1'(b)}\int_{b/6}^{x^{1-c}/12}\frac{1}{a}\mathrm{d}a\\&=\left(\frac{A'_1}{6}+o(1)\right)\sum_{\substack{x^{1-\alpha}(\log x)^{-B}/2<b<x^{1-c}/2\\(b,6)=1}}\frac{H(b)}{b}\log (x^{1-c}/(2b))\\&=\left(\frac{A'_1}{6}\frac{A_1}{2H(6)}+o(1)\right)\int_{x^{1-\alpha}(\log x)^{-B}/2}^{x^{1-c}/2}\frac{\log (x^{1-c}/(2b))}{b}\mathrm{d}b
		\\&=\left(\frac{(\alpha-c)^2}{18A_0'}+o(1)\right)\log^2 x,
	\end{align*}
	where $A_1'A_1=4/(3A_0')$,
	\begin{align*}
		H(n)=\prod_{p>2,p|n}\left(\frac{p-1}{p-2}\right),\quad A_1=\prod_{p>2}\left(1+\frac{1}{p(p-2)}\right).
	\end{align*} Similarly, we have
	\begin{align*}
		&\sum_{\substack{x^{1-\alpha}(\log x)^{-B}/2<b<a<x^{1-c}/2\\6|b, (a,b)=1}}\frac{H_1(ab)}{ab}=\left(\frac{(\alpha-c)^2}{18A_0'}+o(1)\right)\log^2 x,\\&\sum_{\substack{x^{1-\alpha}(\log x)^{-B}/2<b<a<x^{1-c}/2\\2|a,3|b, (a,b)=1}}\frac{H_1(ab)}{ab}=\left(\frac{(\alpha-c)^2}{18A_0'}+o(1)\right)\log^2 x,\\&\sum_{\substack{x^{1-\alpha}(\log x)^{-B}/2<b<a<x^{1-c}/2\\3|a,2|b, (a,b)=1}}\frac{H_1(ab)}{ab}=\left(\frac{(\alpha-c)^2}{18A_0'}+o(1)\right)\log^2 x.
	\end{align*}
	Then (4.16) becomes
	\begin{align*}
		\mathscr{S}_B&\leq (216A_0'+o(1))\frac{x}{((2c-1-\varepsilon)\log x)^3}\sum_{\substack{x^{1-\alpha}(\log x)^{-B}/2<b<a<x^{1-c}/2\\6|ab, (a,b)=1}}\frac{H_1(ab)}{ab}\\&= \left(\frac{48(\alpha-c)^2}{(2c-1-\varepsilon)^3}+o(1)\right)\pi(x).\tag{4.17}
	\end{align*}
	Conclude from (4.13), (4.14) and (4.17) that
	\begin{align*}
		\sum_{\substack{p\leq x\\ P^+_y(p-1)<P^+_y(p+1)}}1\geq (g_3(c,\alpha)+o(1))\pi(x),
	\end{align*}
	where $1/2+\delta'\leq \alpha<17/32$, $1/2+\varepsilon<c<\alpha$,
	\begin{align*}
		g_3(c,\alpha)=C(\alpha)\log\frac{\alpha}{c}-\frac{48(\alpha-c)^2}{(2c-1-\varepsilon)^3}
	\end{align*}
	Consider $1/2+c_0<c<\alpha$, where $c_0$ satisfies
	\begin{align*}
		c_0>0,\quad c_0+\frac{C(\alpha)(2c_0-\varepsilon)^3}{96}<\delta'.
	\end{align*}Then we have
	\begin{align*}
		g_3(c,\alpha)\geq C(\alpha)(\alpha-c)-\frac{48}{(2c_0-\varepsilon)^3}(\alpha-c)^2.
	\end{align*}
	Taking
	\begin{align*}
		\alpha-c=\frac{C(\alpha)(2c_0-\varepsilon)^3}{96},
	\end{align*}
	we have $c\in(1/2+c_0,\alpha)$ and
	\begin{align*}
		g_3(c,\alpha)\geq \frac{C(\alpha)^2(2c_0-\varepsilon)^3}{192}.
	\end{align*}
	
	The above may be summarised as follows: For any $1/2+\delta'\leq \alpha<17/32$, there exists $h_2(\alpha)$ such that
	\begin{align*}
		\sum_{\substack{p\leq x\\ P^+_y(p-1)<P^+_y(p+1)}}1\geq (h_2(\alpha)+o(1))\pi(x),\tag{4.18}
	\end{align*}
	where
	\begin{align*}
		h_2(\alpha)=\frac{C(\alpha)^2(2c_0-\varepsilon)^3}{192}>0.
	\end{align*}
	In particular, the function \(h_2\) is monotonically decreasing.
	
	Combining all the above cases, we complete the proof of Theorem 1.6.

	\section*{Acknowledgements}
	The authors thank Professor Zhiwei Wang for exchanging some information.

\end{document}